\documentclass[12pt, reqno]{amsart}
\usepackage{ifthen,amsfonts,amsmath,amssymb,epic,eepic}
\usepackage{epsfig,color}

\makeatother

\theoremstyle{definition}

\newtheorem{Thm}{Theorem}
\newtheorem{Lem}{Lemma}
\newtheorem{Rem}{Remark}
\newtheorem{Pro}{Proposition}

\numberwithin{equation}{section}

\newcommand{\Vol}{{\text{Vol}}}

\def\RR{{\bold R}}

\newcommand{\e}{{\text {e}}}

\newcommand{\cH}{{\mathcal{H}}}

\newcommand{\eqr}[1]{(\ref{#1})}

\title[Nodal sets]
{Lower bounds for nodal sets of eigenfunctions}

\author{Tobias H. Colding}%
\address{MIT, Dept. of Math.\\
77 Massachusetts Avenue, Cambridge, MA 02139-4307.}
\author{William P. Minicozzi II}%
\address{Johns Hopkins University\\
Dept. of Math.\\
3400 N. Charles St.\\
Baltimore, MD 21218}

\thanks{The   authors
were partially supported by NSF Grants DMS  0606629, DMS
 0906233,  and NSF FRG grants DMS 
 0854774 and DMS 0853501}

\email{colding@math.mit.edu  and minicozz@math.jhu.edu}

\begin{document}

\maketitle

\section{Introduction}

Let $M$ be a smooth closed Riemannian manifold and $\Delta$ the Laplace operator.  A function $u$ is said to be an eigenfunction with eigenvalue $\lambda$ if
\begin{equation}
	\Delta u = - \lambda \, u \, .
\end{equation}
With our convention on the sign of $\Delta$, the eigenvalues are non-negative and go to infinity.

One of the most fundamental questions about eigenfunctions is to understand the sets where they vanish; these sets are called {\emph{nodal sets}}.     Nodal sets   are $(n-1)$-dimensional manifolds away from   $(n-2)$-dimensional singular sets   where the gradient also vanishes, so it is natural to estimate the $(n-1)$-dimensional Hausdorff measure $\cH^{n-1}$.

   \vskip2mm
   The main result of this short note is the following lower bound:
   
   \begin{Thm}	\label{t:main}
   Given a closed $n$-dimensional Riemannian manifold $M$, there exists $C$ so that 
   \begin{equation}
   	\cH^{n-1} (\{ u = 0 \}) \geq C \, \lambda^{\frac{3-n}{4}} \, .
   \end{equation}
   \end{Thm}
   
   \vskip2mm
   In particular, Theorem \ref{t:main} gives a uniform lower bound  in dimension $n=3$.
   
  \vskip2mm
In \cite{Y}, S.-T. Yau conjectured the lower bound
$C \, \sqrt{\lambda}$ in all dimensions.
   This   was proven for surfaces by Br\"uning, \cite{B}, and Yau, independently, and for 
  real analytic metrics by
  Donnelly and Fefferman in \cite{DF}, but remains open in the smooth case.
  The Donnelly-Fefferman argument leads to exponentially decaying  lower bounds in the smooth case; see also \cite{HL}. 
  
    In \cite{M}, Mangoubi considered eigenfunctions on a ball in a manifold and proved   lower bounds for the volume of the subset of the ball where  the function is positive and the subset where it is negative.  Combined with the isoperimetric inequality, one can get the lower bound
 $C \, \lambda^{ \frac{3-n}{2} - \frac{1}{2n}}$ for the measure of the nodal set on the entire manifold.
  
  \vskip2mm
  Recently, Sogge and Zelditch, \cite{SZ}, proved the lower bound $C \, \lambda^{ \frac{7-3n}{8} }$.
Their argument is completely different and is based in part on 
 a beautiful new integral formula relating the $L^1$ norm of $|\nabla u|$ on the nodal set and the $L^1$ norm of $u$ on $M$.

 \vskip1mm
 Finally,  note that some papers use $\lambda^2$ for the eigenvalue (i.e., $\Delta u = - \lambda^2 \, u$); with that convention, our bound is
  $C \, \lambda^{\frac{3-n}{2}}$.
 
 \section{Finding good balls}

Throughout  $M$ will be a fixed closed manifold with Laplace operator $\Delta$ and $u$ will be an eigenfunction with eigenvalue $\lambda$ and $\int_M u^2 = 1$.  We will always assume that $\lambda \geq 1$ since   our interest is in what happens when $\lambda$ goes to infinity.  

\vskip2mm
The first step is to fix a scale $r$ depending on $\lambda$:

\begin{Lem}	\label{l:r}
There exists $a > 0$ so that  $u$ has a zero in every ball of radius $\frac{  a}{3 \,  \sqrt{\lambda }}$.
\end{Lem}

This lemma is standard, but we will include a proof since it is so short.

\begin{proof}
If $u$ does not vanish on $B_{\frac{r}{3}}$, then Barta's theorem implies that the lowest Dirichlet eigenvalue on this ball is at least $\lambda$.  Let $\phi$ be a function that is identically one on $B_{\frac{r}{6}}$ and that cuts off linearly to zero on the annulus $B_{\frac{r}{3}} \setminus B_{\frac{r}{6}}$.  It follows that
\begin{equation}
	\lambda \leq \frac{ \int |\nabla \phi |^2}{\int \phi^2} \leq
	\frac{ 36 r^{-2} \, \Vol ( B_{ \frac{r}{3}}) }{  \Vol( B_{ \frac{r}{6} } )}  \leq \frac{C}{r^2} \, ,
	\end{equation}
	where $C$ comes from the Bishop-Gromov volume comparison (see page $275$ of \cite{G})
	and depends only on $n$, the Ricci curvature of $M$, and an upper bound for $r$.  Since this is impossible for
	$r^2 > C \, \lambda^{-1}$, the lemma follows.
\end{proof}

 From now on,  
  we set $r = a \, \lambda^{ - \frac{1}{2} }$ with $a$ given by Lemma \ref{l:r}.
  
\vskip2mm
Next we use a standard covering argument to decompose the manifold $M$ into small balls of radius $r$.  If $B$ is a ball in $M$, then we write $2B$ for the ball with the same center as $B$ and twice the radius.

\begin{Lem}	\label{l:decomp}
There exists a collection $\{ B_i \}$ of balls of radius $r$ in $M$ so that $M \subset \cup_i B_i$ and each point of $M$ is contained in at most $C_M = C_M (M)$ of the double balls $2B_i$'s.
\end{Lem}

\begin{proof}
Choose a maximal disjoint collection of balls $B_{\frac{r}{2}} (p_i)$.  It follows immediately from maximality that the double balls $B_i = 
B_r (p_i)$ cover $M$.

Suppose that $p \in M$ is contained in balls $B_{2r}(p_1) , \dots , B_{2r}(p_k)$.  In particular, the disjoint balls 
$B_{\frac{r}{2}}(p_1) , \dots , B_{\frac{r}{2}}(p_k)$ are all contained in $B_{3r}(p)$ so that
\begin{equation}	\label{e:BG1}
	\sum_{i=1}^k \Vol (B_{\frac{r}{2}}(p_i)) \leq \Vol (B_{3r} (p)) \, .
\end{equation}
On the other hand, for each $i$, the Bishop-Gromov volume comparison gives $C_M$ that depends only on $n$, an upper bound on $r$, and a bound on the Ricci curvature of $M$ so that
\begin{equation}	\label{e:BG2}
	\Vol (B_{3r} (p)) \leq \Vol ( B_{5r}(p_i)) \leq C_M \, \Vol ( B_{\frac{r}{2}}(p_i)) \, .
\end{equation}
Combining \eqr{e:BG1} and \eqr{e:BG2} gives that $k \leq C_M$.  Since   $p$ is arbitrary, the lemma follows.
\end{proof}

 From now on, we will use the balls $B_i$ given by Lemma \ref{l:decomp}.  These balls will be sorted into two groups, depending on how
 fast $u$ is growing from $B_i$ to $2B_i$.   The two groups will be the ones that are $d$-good and the ones that are not.
 Namely, 
 given a constant $d > 1$, we will say that a ball $B_i$ is $d$-good if 
\begin{equation}	\label{e:dgood}
	\int_{2B_i} u^2 \leq 2^d \, \int_{B_i} u^2 \, .
\end{equation}
Let $G_d$ be the union of the $d$-good balls 
\begin{equation}
G_d = \cup \{ B_i \, | \, B_i {\text{ is $d$-good }} \}  \, .
\end{equation}

The next lemma shows  that most of the $L^2$ norm of $u$ comes from $d$-good balls, provided that $d$ is chosen fixed large independently of $\lambda$.

\begin{Lem}	\label{l:mostgood}
There exists $d_M$ depending only on $C_M$ so that if $d \geq d_M$, then
\begin{equation}	\label{e:mostgood}
	\int_{G_d} u^2 \geq \frac{3}{4} \, .
\end{equation}
\end{Lem}

\begin{proof}
Let $\Omega = \cup \{ B_i \, | \, B_i  {\text{ is not $d$-good }} \}$ be the union of the balls $B_i$ that are {\emph{not}} $d$-good.  
Since the $B_i$'s cover $M$, we have
\begin{equation}	\label{e:omegaG}
	\int_{G_d} u^2 \geq \int_M u^2 - \int_{\Omega} u^2 = 1 - \int_{\Omega} u^2 \, .
\end{equation}
If the ball $B_i$ is not $d$-good, then 
\begin{equation}	\label{e:notgood}
	\int_{2B_i} u^2 >  2^d \, \int_{B_i} u^2 \, .
\end{equation}
Summing \eqr{e:notgood} over the balls that are not $d$-good gives
\begin{equation}	\label{e:notgood2}
	\int_{\Omega} u^2 \leq \sum_{ B_i {\text{ is not $d$-good }} }  \, \, \int_{B_i} u^2 \leq
	\sum_i \, 2^{-d} \, \int_{2B_i} u^2 \leq 2^{-d} \, C_M \, \int_M u^2 = 2^{-d} \, C_M \, ,
\end{equation}
where the second inequality used \eqr{e:notgood} and the third inequality used Lemma \ref{l:decomp}.  If we choose $d_M$ so that
$2^{-d_M} \, C_M = \frac{1}{4}$, then \eqr{e:mostgood} follows by combining \eqr{e:omegaG} and 
\eqr{e:notgood2}.
\end{proof}

To get a lower bound for the number of good balls, we will use the
following $L^p$ bounds for eigenfunctions proven by Sogge in \cite{S1}: 
\begin{equation}	\label{p:big}
	\|  u \|_{L^p} \leq
	\begin{cases}
	C \, \lambda^{ \frac{ n(p-2) - p}{4p} }  &
	\text{ if }
 p \geq \frac{2(n+1)}{n-1} \, , \\
 	C \, \lambda^{ \frac{ (n-1)(p-2)}{8p} } &
	\text{ if }
 p \leq \frac{2(n+1)}{n-1} \, .
\end{cases}
\end{equation}
We will only use this estimate for $p= \frac{2(n+1)}{n-1}$ as this gives the sharpest bound; see the remark right after the proof.

\begin{Lem}	\label{l:somegood}
There exists $C$ depending only on $M$ so that 
there are at least $C \, \lambda^{ \frac{n+1}{4} }$ balls that are $d_M$-good.
\end{Lem}

\begin{proof}
Let $N$ denote the number of $d_M$-good balls $B_i$.  Given any $p>2$, 
the $L^2$ norm of $u$ on a set $G$ is bounded by
\begin{equation}	\label{e:basic}
	\int_{G} u^2 \leq \left( \int_{G} 1 \right)^{ \frac{p-2}{p} }\,
	\left( \int_{G} (u^2)^{\frac{p}{2}} \right)^{\frac{2}{p}}
	\leq  \left( \Vol (G) \right)^{ \frac{ p-2}{p}} \, \| u \|_{L^p}^2  \, .
\end{equation}
Raising both sides to the $\frac{p}{p-2}$ power, bringing the $L^p$ norm to the other side, and setting $G=G_{d_M}$ gives
\begin{equation}	\label{e:basic1}
	  \left(  \frac{3}{4} \right)^{ \frac{p}{p-2} }  \,  \, \| u \|_{L^p}^{-\frac{2p}{p-2} } 
	 	\leq
	  \left(  \int_{G_{d_M}} u^2 \right)^{ \frac{p}{p-2} }  \,  \, \| u \|_{L^p}^{-\frac{2p}{p-2} } 
	 	\leq    \Vol (G_{d_M})  \, ,
\end{equation}
where the first inequality used  Lemma \ref{l:mostgood}.  Thus, for any $p\leq \frac{2(n+1)}{n-1}$, the $L^p$ eigenfunction bound \eqr{p:big} gives
\begin{equation}	\label{e:basic1a}
	  C_p \, \lambda^{ \frac{1-n}{4} } =  C_p \, \left( \lambda^{\frac{(n-1)(p-2)}{8p}} \right)^{-\frac{2p}{p-2} } 
	 	\leq    \Vol (G_{d_M})  \, ,
\end{equation}
where the constant $C_p$ depends on $p$ and $M$ but not on $\lambda$.

Since  $\Vol ( B_i) \leq C_M' \, r^n = C' \, \lambda^{ - \frac{n}{2}}$  for  $C_M'$ and $C'$ depending on $M$ (in fact, just on $n$, a lower bound for the Ricci curvature, and an upper bound on $r$), we get that
\begin{equation}
 C \, \lambda^{ \frac{1-n}{4} }  \leq \Vol (G_d) \leq N \,   C' \, \lambda^{- \frac{n}{2} } \, ,
 \end{equation}
 giving the lemma.
 \end{proof}
 
 \begin{Rem}	
Setting $p= \frac{2(n+1)}{n-1}$ gives the sharpest bound in Lemma \ref{l:somegood}.  
To see this, suppose that $p\geq \frac{2(n+1)}{n-1}$ and
use the $L^p$ eigenfunction bound \eqr{p:big} in \eqr{e:basic1} to get
\begin{equation}	\label{e:basic1b}
	  C_p \, \lambda^{ \frac{p}{2(p-2)} - \frac{n}{2} } =  C_p \, \left( \lambda^{\frac{ n(p-2) - p}{4p}} \right)^{-\frac{2p}{p-2} } 
	 	\leq    \Vol (G_{d_M})  \, ,
\end{equation}
where the constant $C_p$ depends on $p$ and $M$ but not on $\lambda$.   Since $\frac{p}{2(p-2)}$ is monotone decreasing in $p$, the bound \eqr{e:basic1b} is sharpest at the endpoint $p=\frac{2(n+1)}{n-1}$.
 \end{Rem}
 
 \begin{Rem}
 The lower bound \eqr{e:basic1a} for the volume  where   $\int u^2$ concentrates 
 is sharp.  There are spherical harmonics  concentrating on a $\lambda^{ - \frac{1}{4} }$  neighborhood of a geodesic; see \cite{S2}. 
 \end{Rem}
 
 \begin{Rem}
 If we used above the Sobolev inequality (page $89$ in \cite{ScY}) instead of the $L^p$-bounds of Sogge, then we would get the following lower bound for the volume of $G$ where $p=\frac{2n}{n-2}$
 \begin{align}	\label{e:basicremark}
	\int_{G} u^2 &\leq \left( \int_{G} 1 \right)^{ \frac{p-2}{p} }\,
	\left( \int_{G} (u^2)^{\frac{p}{2}} \right)^{\frac{2}{p}}
	\leq  \left( \Vol (G) \right)^{ \frac{ p-2}{p}} \, \| u \|_{L^p}^2\notag\\
	&\leq   \left( \Vol (G) \right)^{ \frac{ p-2}{p}} \, \| \nabla u \|_{L^2}^2=\left( \Vol (G) \right)^{ \frac{ p-2}{p}} \,\lambda\, .
\end{align}
Note that the Sobolev inequality holds since $u$ is an eigenfunction so $\int_M u=0$.  
Since $\frac{p-2}{p}=1-\frac{2}{p}=\frac{2}{n}$, we get
\begin{equation}
\int_{G} u^2 \leq \left(\Vol (G)\right)^{\frac{2}{n}}\,\lambda\, ,
\end{equation}
which gives only that the number of good balls is bounded below (independent of $\lambda$).  This leads to the lower bound $C \, \lambda^{ \frac{1-n}{2} }$ for the measure of the nodal set.
 \end{Rem}

 \section{Local estimates for the nodal set}

 The main theorem will follow by combining the lower bound on the number of good balls with a  
  lower bound for the nodal volume in each good ball.  This local estimate for the nodal set is  based on  the isoperimetric inequality together with estimates for the sets where the function is positive and negative; this approach comes from \cite{DF} where they prove a similar local estimate (cf. also \cite{HL}).
 
 \vskip1mm
 The  local lower bound for nodal volume is the following:
 
 \begin{Pro}	\label{p:local}
 Given  constants $d > 1$ and $\rho >1$, there exist $\mu > 0$ and $\bar{\lambda}$ so that if
 $\Delta u = - \lambda \, u$ on $B_r(p) \subset M$ with $r \leq \rho \, \lambda^{- \frac{1}{2} }$, $\lambda \geq \bar{\lambda}$,
 $u$ vanishes somewhere in $B_{\frac{r}{3}} (p)$, and
 \begin{equation}	\label{e:CD}
 	\int_{B_{2r}(p)} u^2 \leq 2^d \, \int_{B_r(p)} u^2 \, , 
\end{equation}
then 
\begin{equation}	\label{e:locH0}
	\cH^{n-1} \left( B_r(p) \cap \{ u = 0 \} \right) \geq \mu \, r^{n-1} \, .
\end{equation}
\end{Pro}

\vskip2mm
Given this proposition, we can now prove the main theorem:

\begin{proof}
(of Theorem \ref{t:main}).  We can assume that $\lambda$ is large.  By Lemma \ref{l:somegood}, 
 there are at least $C \, \lambda^{ \frac{n+1}{4} }$ balls that are $d_M$-good.  
 
If $B_i$ is any of the balls in the covering, then   Lemma \ref{l:r} implies that $u$ vanishes somewhere in $\frac{1}{3} \, B_i$.  
Thus, if $B_i$ is $d_M$-good, then 
 Proposition \ref{p:local} (with $d=d_M$ and $r= a \, \lambda^{- \frac{1}{2} }$) gives  
 \begin{equation}
 	\cH^{n-1} \left( B_i \cap \{ u = 0 \} \right) \geq C_1 \,\lambda^{ - \frac{n-1}{2}  } \, , 
 \end{equation}
 where $C_1$ depends only on $M$.  Here, we have used that we can assume that $\lambda$ is large.
 
 Combining these two facts and using the covering bound from Lemma \ref{l:decomp} gives
 \begin{equation}
 		C \, \lambda^{ \frac{n+1}{4} } \,  C_1 \,\lambda^{ - \frac{n-1}{2}  } \leq 
		\sum_{B_i {\text{ $d_M$-good}}} \, \cH^{n-1} \left( B_i \cap \{ u = 0 \} \right) 
		\leq C_m \, \cH^{n-1} \left(   \{ u = 0 \} \right) \, ,
 \end{equation}
 and the theorem follows.
\end{proof}

It   only remains to prove the local estimate in Proposition \ref{p:local}.

 \subsection{Proof of the local lower bound}

 In Euclidean space, if an eigenfunction $u$ vanishes at a point $p$, then it's average is zero on every ball with center at $p$.  We will use the following generalization of this:
 
 \begin{Lem}	\label{l:avg}
 There exists $\bar{R} > 0$ depending on $M$ 
 so that if $r \leq \bar{R}$ and $u(p)=0$, then
 \begin{equation}	
 	\left| \int_{B_r(p)} u \right| \leq \frac{1}{3} \, \int_{ B_r(p) } |u| \, .
\end{equation}
 \end{Lem}
 
 \begin{proof}
 Given a function $v$, define the  spherical average
 \begin{equation}
 	I_v (s) = s^{1-n} \, \int_{\partial B_s (p)} v \, .
 \end{equation}
Let $d$ denote the distance in $M$ to $p$.  Differentiating $I_v (s)$ gives
 \begin{align}
 	I_v ' (s) &= 	s^{1-n} \, \int_{\partial B_s (p)} \left[ \frac{\partial v}{\partial r} 
		+ v \, \left( \Delta d + \frac{1-n}{s} \right) \right]  \notag \\
		&=  s^{1-n} \, \int_{B_s (p)} \Delta v + s^{1-n} \, \int_{\partial B_s (p)} 
		v \, \left( \Delta d + \frac{1-n}{s} \right)  \, ,
 \end{align}
 where the second equality used the divergence theorem.  On $\RR^n$, 
 $  \Delta d + \frac{1-n}{d}   = 0$.  On   $M$, as long as $d$ is small (depending on  $\sup |K_M|$ and lower bound for the injectivity radius of $M$), the Hessian comparison theorem (see, e.g., page $4$ in \cite{ScY}) gives
 \begin{equation}	\label{e:hessian}
 	\left| \Delta d + \frac{1-n}{d} \right| \leq h (d) \, ,
 \end{equation}
 where the function $h:[0,\infty) \to \RR$ is continuous, monotone non-decreasing, and satisfies $h(0)=0$.  Thus, we see that
 \begin{align}	\label{e:diffIu}
 	\left| I'_u (s)\right|  \leq  \lambda \, s^{1-n} \, \left| \int_{B_s (p)}  u \right|+ 
	s^{1-n} \, h(s) \, \int_{\partial B_s (p)} 
		|u| 
		\leq  \frac{\lambda \, s}{n} \,   \max_{t\leq s} \, |I_u (t)| + 
	 h(s) \, I_{ 
		|u|}(s)
		 \, .
 \end{align}
 Motivated by this, define the function $f$ by
 \begin{equation}
 	f(s) =  \max_{t\leq s} \, |I_u (t)| \, .
\end{equation}
 Observe that $f$ is automatically monotone non-decreasing,   $f(0)=0$ (since $u(p)=0$), and $f$ is Lipschitz with
 \begin{equation}
 	f'(s) \leq \left| I'_u (s)\right| \leq  \frac{\lambda \, s}{n} \,   f(s) + h(s) \, I_{ 
		|u|}(s)
	 \, ,
 \end{equation}
 where this inequality is understood in the sense of the limsup of forward difference quotients.  In particular, we have for $s \leq r$ that
  \begin{equation}
 	\frac{d}{ds} \, \left( f(s) \, \e^{ - \frac{\lambda \, r \, s}{n} } \right) \leq   h(s) \, I_{ 
		|u|}(s)
	 \, .
 \end{equation}
Using that $f(0)=0$ and integrating this gives for each $t \leq r$ that
  \begin{equation}
 	 f(t) \leq  \e^{  \frac{\lambda \,r \, t}{n} }  \,   h(t) \, \int_0^t I_{ 
		|u|}(s) \, ds \leq  \e^{  \frac{\lambda \, r^2}{n} }  \,   h(r) \, \int_0^t I_{ 
		|u|}(s) \, ds
	 \, .
 \end{equation}
 By the coarea formula, we have
\begin{equation}
 	\left| \int_{B_r(p)} u \right| = \left| \int_0^r t^{n-1} \,  I_u (t)   \, dt \, \, \right|
	\leq \frac{r^n}{n} \, f(r) \leq \,
	\e^{  \frac{\lambda \, r^2}{n} }  \,   h(r) \,  \frac{r^n}{n} \, \int_0^r I_{ 
		|u|}(t) \, dt \, .
	\end{equation}
	Observe that $ \e^{  \frac{\lambda \, r^2}{n} }$ is bounded since $r$ is on the order of $\lambda^{-1/2}$ and we can make  $h(r) $ as small as we like by taking $r$ small enough (independent of $\lambda$).  Thus, to finish off the proof, we need only bound $r^n \, \int_0^r I_{ 
		|u|}(t) \, dt$ by a fixed multiple of $\int_{B_r(p)} |u|$.  To do this, observe that, since  
		$r$ is proportional to  $\lambda^{-1/2}$,  the mean value inequality (theorem $1.2$ in 
		\cite{LiSc}) gives
		\begin{equation}
			\sup_{B_{r/2}(p)} |u| \leq C \, r^{-n} \, \int_{B_r(p)} |u| \, ,
		\end{equation}
		so we get
	\begin{align}
		  \int_0^r I_{ 
		|u|}(t) \, dt  &=    \int_0^{\frac{r}{2} }  \left( t^{1-n} \, \int_{\partial B_t (p)} |u| \right) \, dt 
		+ \int_{\frac{r}{2}}^r   \left( t^{1-n} \, \int_{\partial B_t (p)} |u| \right) \, dt \notag \\
		&\leq  \frac{r}{2} \, \left( \sup_{B_{r/2}(p)} |u| \right) \, \sup_{t\leq \frac{r}{2}} \,
		\frac{ \Vol (\partial B_t (p))}{t^{n-1}}   + \left( \frac{2}{r} \right)^{n-1} \, 
		  \int_{\frac{r}{2}}^r     \int_{\partial B_t (p)} |u|   \, dt \notag \\
		&\leq C_1 \, r^{1-n} \, \int_{B_r(p)} |u|
		\, .
	\end{align} 
	Thus, since $r\leq \bar{R}$, we get the desired bound and the lemma follows.
 
 \end{proof}
 
 We are now ready to prove the local lower bound:

\begin{proof}
(of Proposition \ref{p:local}).  Let $q \in B_{\frac{r}{3}}(p)$ be a point with $u(q)=0$.  Note that
\begin{equation}
	B_{r}(p) \subset B_{\frac{4r}{3}}(q) {\text{ and }}
	B_{\frac{5r}{3}}(q) \subset B_{2r}(p) \, .
\end{equation}
Since the scale $r$ is proportional to $\lambda^{- \frac{1}{2} }$, we can apply the meanvalue inequality (theorem $1.2$ in 
		\cite{LiSc}) to $u^2$ to get
\begin{equation}	\label{e:sup}
	\sup_{B_{\frac{4r}{3}}(q)} u^2 \leq C_0 \, r^{-n} \, \int_{B_{2r}(p)} u^2 
	\leq C_0 \, 2^d \, r^{-n} \, \int_{B_{r}(p)}  u^2
	\leq C_0 \, 2^d \, r^{-n} \, \int_{B_{\frac{4r}{3}}(q)}  u^2 \, ,
\end{equation}
where the second inequality used
 \eqr{e:CD} and $C_0$ depends only on $n$,   the geometry of $M$, and
an upper bound for $r^2 \, \lambda$   (all of which are fixed).
 
  From now on, all integrals will be over $B_{\frac{4r}{3}}(q)$ unless stated otherwise. Using \eqr{e:sup}, we get the ``reverse H\"older'' inequality
  \begin{equation}
  	\left( \int u^2 \right)^2  \leq \sup u^2 \, \left( \int |u| \right)^2 \leq 
	 C_0 \, 2^d \, r^{-n} \, \left( \int  u^2 \right) \,  \left( \int |u| \right)^2 \, ,
  \end{equation}
  which simplifies to
  \begin{equation}	\label{e:rH}
  	\int u^2 \leq  C_0 \, 2^d \, r^{-n} \,  \left( \int |u| \right)^2 \, .
  \end{equation}

 Let $u^+$ be the positive part of $u$, i.e., $u^+ (x) = \max \{ u(x) , 0 \}$, and let $u^- = u^+ - u$ be the negative part of $u$.
 It follows from Lemma \ref{l:avg} that
 \begin{equation}
 	\int u^+ \geq \frac{1}{3} \, \int |u| {\text{ and }} 
		\int u^- \geq \frac{1}{3} \, \int |u| \, .
\end{equation}
Let $B^+$ denote $B_{\frac{4r}{3}}(q) \cap \{ u > 0 \}$ and  $B^-$ denote $B_{\frac{4r}{3}}(q) \cap \{ u < 0 \}$.
Thus, applying Cauchy-Schwarz to $u^+$ gives
\begin{equation}
	\frac{1}{9} \, \left( \int |u| \right)^2 \leq \left( \int u^+ \right)^2 \leq \Vol ( B^+ )  \, \int u^2 \leq \Vol ( B^+ ) \, C_0 \, 2^d \, r^{-n} \, 
	 \left( \int |u| \right)^2 \, ,
\end{equation}
where the last equality used \eqr{e:rH}.  Dividing through by the square of the $L^1$ norm of $u$ gives a scale-invariant lower bound for the volume of $B^+$
\begin{equation}
	\frac{r^n}{9 \, C_0 \, 2^d }  \leq  \Vol ( B^+ )  \, .
\end{equation}
The same argument applies to $u^-$ to give the same lower bound for the volume of $B^-$.  
Together, these allow us to apply the isoperimetric inequality to get the lower bound for the measure of the nodal set in $B$, thus completing the proof of the proposition.
 \end{proof}

\end{document}